\documentclass[smallextended,envcountsect,envcountsame]{svjour3}
\usepackage{graphicx}
\usepackage{mathrsfs}
\usepackage{fix-cm}
\usepackage{amsmath}
\usepackage{amssymb}
\usepackage{latexsym}
\usepackage{dsfont}
\usepackage{xcolor}
\usepackage{cite}
\usepackage{dsfont}
\usepackage{enumitem}
\usepackage{amsmath, amssymb}

\def\H{H}
\def\A{\mathcal{A}}
\def\X{X}
\def\Y{Y}
\def\tto{\rightrightarrows}

\def\B{\mathbb{B}}

\def\R{\mathbb{R}}
\def\N{\mathbb{N}}

\def\co{\operatorname{co}}
\def\gph{\operatorname{gph}}
\def\epi{\operatorname{epi}}
\def\dom{\operatorname{dom}}
\def\cl{\operatorname{cl}}
\def\bd{\operatorname{bd}}

\def\N{\mathbb{N}}

\def\tt{\tilde t}

\def\dmu{\mu(d\omega)}

\newcommand{\MoreauYosida}[2]{\operatorname{e}_{#2} #1}

\newcommand{\IntfLp}[1]{{\mathcal{I}}_{#1}}

\DeclareMathOperator{\inte}{\operatorname{int} }

\newcommand{\Rex}{\overline{\mathbb{R}}}
\let\epsilon\varepsilon

\usepackage[colorinlistoftodos]{todonotes}

\title{Random multifunctions as the set minimizers of infinitely many differentiable random functions
	\thanks{This work was partially supported by ANID-Chile under  grants Fondecyt Regular 1190110, Fondecyt Regular 1200283 and Fondecyt de Iniciaci\'on 11180098.}}
\begin{document}

\titlerunning{Random multifunctions as minimizers  differentiable random functions}
\author{Juan Guillermo Garrido \and  Pedro P\'erez-Aros \and Emilio Vilches}

\institute{Juan Guillermo Garrido  \at Departamento de Ingenier\'ia Matem\'atica,  Universidad de Chile, Santiago, Chile\\
\email{ jgarrido@dim.uchile.cl } \and 
Pedro P\'erez-Aros \and Emilio Vilches \at Instituto de Ciencias de la Ingenier\'ia, Universidad de O'Higgins, Rancagua, Chile\\
\email{pedro.perez@uoh.cl,\,emilio.vilches@uoh.cl}} 


\maketitle

\begin{abstract}
Under mild assumptions, we prove that any random multifunction can be represented as the set of minimizers of an infinitely many differentiable normal integrand, which preserves the convexity of the random multifunction.  
We provide several applications of this result to the approximation of random multifunctions and integrands. The paper ends with a characterization of the set of  integrable selections of a measurable multifunction as the set of minimizers of an infinitely many differentiable integral function.


\end{abstract}

\section{Introduction}

Variational Analysis offers a rich theory to study generalized differentiation of mappings (functions and multifunctions).  The classical way in which the theory is built is to begin with the study of properties of tangents and normal of sets/set-valued maps and then turn to the study of functions through the epigraph/graph of the functions/multifunctions. In this approach, the normal cone plays a prominent role since it allows defining the notion of subdifferential when applied to the epigraph of a given function (see, e.g., \cite{MR1491362} for more details).  Thus, it is of importance to have rules for the calculus of normal cones. Most formulas for such normal cones depends on having an appropriate representation of the set/set-valued map. For instance,  when the set is the sublevel of some smooth function, the normal cone is obtained, under some assumptions, as the cone generated by the gradients of the function at a given point.  A similar result holds for sets defined a the set of minimizers of a smooth function.  Therefore, it is quite desirable to have a representation of a set/set-valued map as the set of minimizers of a smooth function. Unfortunately, it is well-known that not all sets in an arbitrary Banach space can have this representation; consider, e.g., the Banach space of continuous functions over an uncountable set, then a singleton in this space cannot be represented as the set of minimizers of a twice differentiable  function (see, e.g., \cite{MR1622271}).  However, it was proved in  \cite[Theorem~1]{MR1920049} that any closed convex subset of a separable Banach space could be represented as the set of minimizers of an infinitely many differentiable convex function.

In this paper, motivated by the work of Azagra and Ferrera \cite{MR1920049}, we prove, under mild assumptions, that any  random multifunction can be represented as the set of minimizers of an infinitely many differentiable normal integrand. Moreover, this normal integrand can be  constructed in such a way as to preserve the convexity of the random multifunction (Theorem \ref{MainTheorem}). 
We provide several applications of this result to the approximation of random multifunctions and integrands. Moreover, we obtain a characterization of the set integrable selections of a measurable multifunction as the set of minimizers of an infinitely many differentiable integral function and we give some applications to optimization theory.

The paper is organized as follows. After some preliminaries, in Section \ref{MainSECT}, we establish the main result of the paper (Theorem \ref{MainTheorem}), that is, the characterization of random multifunctions  as the set of minimizers of a smooth normal integrand. This result is used to provide a generalization of \cite[Theorem 1]{MR1920049}, that is, a characterization of measurable multifunctions with closed values, as the set of minimizers of an infinitely many differentiable convex normal integrand (see Corollary \ref{mainconvex}). Then,  in Section \ref{Approxima}, we show the existence of smooth approximations of random multifunctions and normal integrands. Section \ref{integrableselect}, provides a characterization of the set of integrable selections of measurable multifunction as the set of minimizers of an infinitely many smooth integral functional. Section \ref{SectionExamples} illustrates how our results can be used to regularize challenging optimization problems. The paper ends with some conclusions and concluding remarks.

\section{Mathematical preliminaries}

In the following $(\X,\Vert \cdot \Vert)$ and $(\H,\langle \cdot , \cdot \rangle)$ will denote a separable Banach space and a separable Hilbert space, respectively. For a set $A$ we denote by $\inte A$, $\cl A$ and $\bd{A}$, the interior, closure and the boundary of $A$, respectively.

Given a nonempty set $S$ of $X$, the support function  and the distance function are defined, respectively, as 
\begin{equation*}
    \sigma_S(x^{\ast}):=\sup_{x\in S}\langle x^{\ast},x\rangle, \quad d(x,S):=\inf_{y\in S} \Vert x-y\Vert,
\end{equation*}
when $S$ is a subset of $X^\ast$, the support function   is defined similarly.

\noindent Given a set $Y$, the graph of a set-valued map $M\colon \X \tto \Y$ is the set
$$
\gph{M}:=\{ (x,y)\in \X\times \Y \colon y\in M(x)\}.
$$
Moreover, for a set $O\subset \Y$, we denote
\begin{equation*}
    \begin{aligned}
    M^-(O):= \{ x\in \X : M(x)\cap O\neq \emptyset \}, \quad  M^+(O):= \{ x\in \X: M(x)\subset O\}.
    \end{aligned}
\end{equation*}
When $\X$ and $\Y$ are two topological spaces, we say that $M$ is \emph{usc} (resp. \emph{lsc}) if $M^{+}(O)$ (resp. $M^{-}(O)$) is open for every open set $O\subset \Y$. If $\Y$ is a Banach space endowed with the weak topology, then we say that $M$ is $\Vert \cdot \Vert$-weak usc (resp. lsc).

Throughout the article  $(\Omega,\mathcal{A},\mu)$ will be a complete $\sigma$-finite measure space. The Banach space of $p$-integrable functions $x:\Omega \to X$ is denoted by $L^p(\Omega, X)$. We say that a set-valued map $M\colon \Omega \tto X$ is \emph{measurable} if $M^{-}(O)\in \mathcal{A}$ for every open set $O\subset X$. Besides, $M$ is \emph{graph-measurable} if its graph  belongs to $\mathcal{A}\otimes\mathcal{B}(X)$, where  $\mathcal{B}(X)$ is the Borel $\sigma$-algebra on $X$.  

Let $X$ be a complete separable metric space, for a multifunction $M\colon \Omega \tto X$ with nonempty and closed values the following assertions are equivalent (see, e.g., \cite{MR0467310,MR2527754,MR2458436}):
\begin{enumerate}[label=\roman*)]
   \item  $M$ is measurable.
    \item  For all $x\in X$, the map $\omega\mapsto d(x,M(\omega))$ is measurable.  
    \item  There exists a sequence of measurable selections $(f_k)_{k}$ of $M$ such that $\{f_k(\omega)\}_k$ is dense in $M(\omega)$ for all $\omega\in \Omega$.
    \item   $M$ is graph measurable, provided that the space $(\Omega, \mathcal{A}, \mu)$ is complete. 
\end{enumerate}
 
A function $f:\Omega\times X\to \overline{\R}$ is called a \emph{normal integrand} if the set-valued mapping $\omega  \tto \epi f (\omega,\cdot) :=\{ (x,\alpha) \in X\times\R \colon  f(\omega, x) \leq \alpha\}$ is measurable with closed values, which is equivalent (under the completeness of $(\Omega, \mathcal{A},\mu)$) to have that for every fixed $\omega\in\Omega$, the map $x\mapsto f(\omega,x)$ is lower semi-continuous and $f$ is $\mathcal{A}\otimes\mathcal{B}(X)$-measurable. 

Given a normal integrand $f:\Omega \times X\to [0,+\infty)$,  an  {integral functional} defined on  $L^p(\Omega,X)$  is  the function $\IntfLp{f}^{\mu,p}: L^p(\Omega,X) \to  [0,+\infty]$, defined by 
\begin{equation*}
  \IntfLp{f}^{\mu,p}(x):=\int_\Omega f(\omega,x(\omega)) d\mu(\omega).
     \end{equation*}
  For simplicity of the notation, we write $\IntfLp{f}$ when there is no ambiguity on the measure space.

 \noindent A topological space $S$ is called \emph{Suslin space} if there exist a Polish space $Y$ (complete and separable metric space) and a continuous function $f:Y\to S$ such that $f(Y) = S$. It is well known that a Polish space is Suslin, and any Suslin space is separable. Moreover, for  a separable Banach space $X$, the dual space  $X^{\ast}$ endowed with the weak$^{\ast}$ topology is  Suslin.

 \noindent The following result is known as the Yankov-von Neumann-Aumann selection theorem (see, e.g., \cite[Theorem III.22]{MR0467310} and \cite[Remark 6.3.21 p.482]{MR2527754}).
\begin{proposition}[Graph measurable selection theorem]\label{graph:measurable:selection:theorem}
	Let $(\Omega, \mathcal{A},\mu)$ be  a  complete $\sigma$-finite measure space and $S$ a Suslin space. Let $M :\Omega \tto  S$ be a graph measurable multifunction with nonempty values. Then, there exists a sequence of measurable selections $m_k : \Omega \to S$ such that 
	\begin{equation*}
	    \{ m_k(\omega)\}_{k} \textrm{ is dense in  } M(\omega) \textrm{ for all } w\in\Omega.
	\end{equation*}
\end{proposition}

\noindent For $p\in [1,\infty]$, the set of $p$ integrable selections of $M$ is denoted by $S_M^p$.

\noindent We say that $M: \Omega \times H \tto X$ is a \emph{random multifunction} or \emph{random set-valued map} if the map $\omega \tto \gph M_\omega$ is measurable and has closed values.   \newline \noindent Here and throughout the paper,  $M_\omega$ will denote the map $x\tto M(\omega, x)$.


Let $X$ be a Banach space. Given a sequence of sets $(A_k)_k\subset X$, the inferior and superior limit of $(A_k)$ are defined, respectively, as 
\begin{equation*}
\begin{aligned}
\limsup A_k&=\{ x\in X\colon \liminf d(x,A_k)=0\}, \\
 \liminf A_k&=\{ x\in X\colon \limsup d(x,A_k)=0\}, 
\end{aligned}
\end{equation*}
Moreover, a sequence of sets $(A_k)\subset X$ \emph{Painlev\'e-Kuratowski converges} to $A$ if $\limsup A_k = \liminf A_k = A$. If $X$ has finite-dimension, then the Painlev\'e-Kuratowski convergence is metrizable over the space of  nonempty and closed sets by the so-called \emph{(integrated) set distance}: 
\begin{equation}\label{metric0}
    \bar{d}(A,B) = \int_0^{+\infty} e^{-\rho} d_\rho(A,B)\text{d}\rho,
\end{equation}
where $A,B\subset X$ and $ \displaystyle d_\rho(A,B) := \max_{x\in \rho \mathbb{B}}   |d(x,A)-d(x,B)|$ for $\rho\geq 0$. We refer to \cite[cap 4.I]{MR1491362} for more details.
 
Now, we introduce the notion of essentially uniformly convergence of multifunctions.  A sequence of measurable multifunctions $(M_k\colon \Omega \tto \R^d)_{k}$ with closed and nonempty values converges \emph{essentially uniformly} to $M\colon \Omega \tto \R^{d}$  respect to the set distance if 
   \begin{equation}\label{metric1}
    \|\bar{d}(M_k,M)\|_\infty :=\inf\{ t: t\geq\bar{d}(M_k(\omega),M(\omega)) \textrm{ for a.e. } w\in \Omega \} \to  0,
   \end{equation}
   as $k\to +\infty$.  In the latter case, the multifunction $M$ must have nonempty values and closed values. 
   
\noindent The above  notion  induces a convergence of functions through their epigraphs.  A sequence of  normal integrands $(f^k\colon \Omega \times \R^d \to \overline{\R})_{k}$ converges \emph{epigraphically uniformly} to $f\colon \Omega\times \R^d\to\overline{\R}$ if the sequence $(\omega\tto \epi f^k_\omega)_k$ converges essentially uniform to $\omega\tto \epi f_\omega$.


\section{Random multifunctions as minimizers of infinitely many differentiable random functions}\label{MainSECT}

This section presents the main result of the paper. We prove that a random multifunction satisfying certain continuity properties can be described as the set of minimizers of a smooth normal integrand. The continuity property required to prove the main theorem is given in the following definition.
\begin{definition}[pseudo-norm-weak usc]\label{pseudo-uppersemicontinuous} A multifunction $M\colon H\tto X$ is said  \emph{pseudo-norm-weak upper-semicontinuous} at $x\in H$ if for all $\alpha \in \mathbb{R}$ and $y^\ast\in X^\ast$ with  $M(x)\subset \{ u\in X \colon  \langle y^\ast,u\rangle <\alpha \}$ and $\eta>0$, there exists $\varepsilon >0$ such that $$M(x')\subset \{u\in X \colon \langle y^\ast,u\rangle<\alpha+\eta\} \textrm{ for all } x'\in \B_{\varepsilon}(x).$$
Moreover, $M$ is said pseudo-norm-weak usc if the above property is satisfied at every point $x\in H$.
\end{definition}
The above notion is weaker than the usual notion of upper semicontinuity for multifunctions when $X$ is endowed with the weak topology. Indeed, the multifunction $M\colon \R\tto \R^2$ defined by $M(t)=t(1,0)+C$, where $C=\{(x,y)\in \R^2\colon xy\geq 1, x\geq 0\}$, is pseudo-norm-weak usc but it is not upper semicontinuous.

The following result, which is the main result of the paper, enables to represent random multifunctions as the level set of a smooth normal integral. 
\begin{theorem}\label{MainTheorem}
	Let  $(\Omega,\A,\mu)$ be a complete $\sigma$-finite measure space, $\H$ be a separable Hilbert space and $X$ be a separable Banach space. If $M : \Omega \times \H \tto X$  is a random multifunction such that  for all $\omega\in \Omega$, $M_\omega\colon H\tto X$ is pseudo-norm-weak upper semicontinuous with convex values, then there exists a normal integrand $\varphi: \Omega\times \H\times X\to [0,+\infty)$ such that 
	\begin{enumerate}[label=(\alph*), ref=\ref{MainTheorem}-(\alph*)]
		\item\label{MainTheorem_a}  For all $\omega \in \Omega$, $\gph M_\omega=\{(x,y)\in H\times X\colon  \varphi_\omega(x,y) = 0\}$.		\item \label{MainTheorem_b}  For all $\omega\in \Omega$, the map $(x,y)\mapsto \varphi_\omega(x,y)$ is $C^\infty$.
		\item \label{MainTheorem_c}  For all $(\omega,x)\in\Omega\times \H$, the map $y\mapsto \varphi_\omega(x,y)$ is convex.
		\item  \label{MainTheorem_d}  There exists $ L\geq 0$ such that for all  $(\omega,y)\in \Omega\times X$, the map $x\mapsto \varphi_\omega(x,y)+ L(\|y\|+1)\|x\|^2$ is convex.
			\item\label{MainTheorem_e}    For all $k\in\N^{\ast}$, there are positive constants $C_k,R$ such that for all $(x,y)\in H\times X$  
\begin{align}
	\sup_{\omega\in \Omega} \|D^k \varphi_\omega (x,y)\| &\leq  \,C_k( \|y\| +1) (\| x \|^k +1 ), \label{eqder1}\\
		\sup_{\omega\in \Omega} \|D_y^k \varphi_\omega (x,y)\| &\leq  \,R(\|x\|+1), \label{eqder2}
\end{align}
where $D^k\varphi_\omega$ and $D^k_y\varphi_\omega$ denote, respectively,  the $k$-derivative and the $k$-derivative with respect to $y$.
\item \label{MainTheorem_f}   For all $\omega\in \Omega$, the map $(x,y)\mapsto \varphi_\omega(x,y)$ satisfies the following continuity property: if $x_n \to x$ and $y_n \rightharpoonup y$, then $\varphi_{\omega}(x_n,y_n) \to \varphi_{\omega}(x,y)$. 
	\end{enumerate}	  
\end{theorem}
\begin{proof}
For the measurable set  $\Omega_\emptyset:=\{  \omega \in \Omega \colon  \gph M_\omega = \emptyset\}$, we can set $\varphi_\omega\equiv 1$ for all $\omega \in \Omega_\emptyset$. Thus, without loss of generality, we can assume  that $\Omega_\emptyset =\emptyset$.\newline
\noindent From now on,  we consider $X^{\ast}$ endowed with the weak$^\ast$-topology. Since $X$ is a separable Banach space, the space $X^{\ast}$ is separable and, hence, a Suslin space. It is worth noting that, when $X^\ast$ endowed with the usual topology is separable, we can proceed with that topology.\newline
\noindent The rest of the proof is divided into several claims. \newline
\noindent {\bf Claim~1:} \emph{The mapping $\mathcal{F}:\Omega\tto H \times X\times  (0,+\infty) \times X^\ast \times \R$ defined by 
\begin{eqnarray}
	\nonumber \mathcal{F}(\omega) &:=& \{(x,y,\varepsilon,y^\ast,\alpha)\in H\times X\times (0,\infty)\times X^\ast\times\R\colon\\
	\nonumber & & \gph M_\omega\cap(\B_\varepsilon(x)\times \{u\in X \colon \langle y^\ast,u\rangle<\alpha\})=\emptyset\textrm{ and }\langle y^\ast,y\rangle<\alpha\},
\end{eqnarray}
has nonempty values and measurable graph.} \newline
\noindent \emph{Proof of Claim 1}:   On the one hand, it is clear that $(0,0,1,0,-1)\in\mathcal{F}(\omega)$ for all $\omega \in\Omega$. Thus, $\mathcal{F}$ has nonempty values. On the other hand, to prove that $\mathcal{F}$ has measurable graph, let us consider  the multifunction $G: H \times (0,+\infty) \times X^\ast \times \R \tto H\times X$ defined by
\begin{align*}
    G(x,\varepsilon, y^\ast,\alpha):=   (\B_\varepsilon(x)\times \{u\in X \colon  \langle y^\ast,u\rangle<\alpha\})^c.
\end{align*}
We observe that 
$$G(x,\varepsilon, y^\ast,\alpha) =  G_1(x,\varepsilon) \times X\, \cup\, H \times G_2(y^\ast,\alpha),$$
with $G_1(x,\varepsilon):=\{a \in H\colon\| x - a\| \geq \varepsilon\}$  and $G_2(y^\ast, \alpha):= \{  u \in X\colon \langle y^\ast , u\rangle \geq \alpha\}$. 
Then, for every open set $U\subset H$, we obtain that 
\begin{align*}
G_1^{-}(U)=\{ (x,\varepsilon)\in H\times (0,\infty): \textrm{there exists }  a\in U \textrm{ such that }  \| x -a\| \geq \varepsilon\}.
\end{align*}
Hence, by taking  $(a_k)_{k\in \N} \subset U$ dense, we obtain that  
\begin{align*}
G_1^{-}(U)=\bigcup_{ k\in \N} \{ (x,\varepsilon)\in H\times (0,\infty) : \| x -a_k \| \geq \varepsilon\},
\end{align*}
which implies that $G_1$ is measurable. 
Similarly, $G_2$ is measurable. Hence,  $G$ is measurable.   Now, let us notice that 
\begin{align*}
 \gph M_\omega \cap  (\B_\varepsilon(x)\times & \{u\in X \colon  \langle y^\ast,u\rangle<\alpha\} )= \emptyset\\
&\Leftrightarrow  \gph(M_\omega) \subset G (x,\varepsilon, y^\ast,\alpha) \\
& \Leftrightarrow  d((v,w), G (x,\varepsilon, y^\ast,\alpha)) \leq 0 \text{ for all } (v,w)  \in  \gph M_\omega\\
& \Leftrightarrow  d((v_k(\omega), w_k(\omega)), G (x,\varepsilon, y^\ast,\alpha)) \leq 0 \text{ for all } k\in \N,
\end{align*}
where $(v_k, w_k)$ is a dense sequence of measurable selections of $\gph M_\omega$. Moreover, the map  $(v,w)\mapsto d((v,w),G (x,\varepsilon, y^\ast,\alpha))$ is continuous and $(x,\varepsilon, y^\ast,\alpha)\mapsto d((v,w),G (x,\varepsilon, y^\ast,\alpha))$ is  $\mathcal{B}(H\times (0,\infty)\times X^\ast\times \R)$ measurable. Hence, for all $k\in \N$, the map $(\omega,x,\varepsilon, y^\ast,\alpha)\mapsto d((v_k(\cdot),w_k(\cdot)), G(\cdot))$ is $\mathcal{A}\otimes \mathcal{B}(H\times (0,\infty)\times X^\ast\times \R )$ measurable. Therefore, the graph of $\mathcal{F}$ can be rewritten as
\begin{align*}
\gph(\mathcal{F}) =\{ (\omega, x,y,\varepsilon,y^\ast,\alpha) : \Phi(\omega, x,\varepsilon,y^\ast,\alpha) \leq 0 \text{ and } \langle y^\ast,y\rangle<\alpha \},
\end{align*}
where $\displaystyle \Phi(\omega, x,\varepsilon,y^\ast,\alpha) :=\sup_{k\in\N} d(v_k(\omega), w_k(\omega), G (x,\varepsilon, y^\ast,\alpha))$ and from here we conclude that $\gph(\mathcal{F})$ is a measurable set.\\

\noindent {\bf Claim~2:} {\em There exist $x_k:\Omega\to H$, $\varepsilon_k:\Omega\to (0,\infty)$, $y_k^\ast:\Omega\to X^\ast$, $\alpha_k:\Omega\to \R$ measurable functions such that  for all  $\omega\in \Omega$
\begin{equation}\label{Omega-incl}
 (\gph M_\omega)^c = \bigcup_{k\in\N} \B_{\varepsilon_k(\omega)}(x_k(\omega))\times \{u\in X \colon \langle y_k^\ast(\omega),u\rangle<\alpha_k(\omega)\}.
 \end{equation}
}	\noindent \emph{Proof of Claim 2}: By virtue of of Claim 1 and Proposition \ref{graph:measurable:selection:theorem},  there exists a sequence of measurable selections $(x_k,y_k,\varepsilon_k,y^\ast_k,\alpha_k)$  of $\mathcal{F}$ such that
$$
(x_k(\omega),y_k(\omega),\varepsilon_k(\omega),y^\ast_k(\omega),\alpha_k(\omega)) \textrm{ is dense  in } \mathcal{F}(\omega) \textrm{ for all   } \omega\in \Omega. 
$$
We proceed to prove that the sequence $(x_k,y_k,\varepsilon_k,y^\ast_k,\alpha_k)$ satisfies \eqref{Omega-incl}.  \newline 
\noindent Indeed,  on the one hand, the inclusion $\supset$ follows by construction. On the other hand, to prove the inclusion $\subset$, take $(x,y)\notin \gph M_\omega$, i.e.,  $y\notin M_\omega(x)$. \newline
\noindent  If $M_\omega(x) = \emptyset$, we can take any $y^\ast\in X^\ast$ and $\alpha\in\R$ such that $\langle y^\ast ,y\rangle<\alpha$, and the neighborhood $\mathcal{U} := \{u\in X\colon \langle y^\ast,u\rangle> \alpha+\xi\}$ for any $\xi >0$. Thus,  $M_\omega(x)\subset \mathcal{U}$ and by the pseudo-norm-weak upper semicontinuity,  there exists $\varepsilon >0$ such that for all $x'\in \B_{\varepsilon}(x)$ the following inclusion holds:
$$M_\omega(x')\subset \{u\in X\colon  \langle y^\ast,u\rangle + \xi>\alpha+\xi\}.$$ Hence, it follows that 
$$\gph(M_\omega)\cap (\B_{\varepsilon}(x)\times\{u\in X \colon \langle y^\ast,u\rangle< \alpha\}) = \emptyset.$$
\noindent  If $M_\omega(x)\neq \emptyset$, by the Hahn-Banach theorem ($M_\omega(x)$ is closed and convex), there exist $y^\ast\in X^\ast$ and $\beta>\alpha$ that 
$$ \langle y^\ast,z\rangle\geq \beta, \textrm{ for all } z\in M_\omega(x) \text{ and } \langle y^\ast,y\rangle<\alpha.$$ 
Moreover, there exists $\xi>0$ such that $\beta >\xi>\alpha$.  Consider the neighbourhood $\mathcal{U} := \{u\in X \colon \langle y^\ast,u\rangle>\xi\}$. Then,  $M_\omega(x)\subset \mathcal{U}$ and by the pseudo-norm-weak upper-semicontinuity of $M_\omega$, there exists $\varepsilon >0$ such that for all $x'\in \B_{\varepsilon}(x)$ the following inclusion holds:
$$M_\omega(x')\subset \{u\in X \colon \langle y^\ast,u\rangle + (\xi-\alpha)>\xi\}.$$  Thus,  $ \gph(M_\omega)\cap (\B_{\varepsilon}(x)\times \{u\in X \colon \langle y^\ast,u\rangle<\alpha\}) = \emptyset $.  \newline 
Therefore, in any case,  $(x,y,\varepsilon,y^\ast,\alpha)\in \mathcal{F}(\omega)$ and $\langle y^\ast,y\rangle<\alpha$. Let us consider $\delta := \alpha-\langle y^\ast,y\rangle>0$. Then, there exists $j\in\N$ such that $$(x_j(\omega),y_j(\omega),\varepsilon_j(\omega),y_j^\ast(\omega),\alpha_j(\omega))\in \mathcal{F}(\omega), $$
and
$$\max\{ \|x-x_j(\omega)\|,|\varepsilon_j(\omega)-\varepsilon|\}<\varepsilon/2,  \max\{|\alpha_j(\omega)-\alpha|, |\langle y^\ast_j(\omega)-y^\ast,y\rangle|\}<\delta/2.$$ Then, since $\|x-x_j(\omega)\|\leq \varepsilon/2 <\varepsilon_j(\omega)$, we obtain that  $x\in \B_{\varepsilon_j(\omega)}(x_j(\omega))$.
\noindent Moreover,  
\begin{eqnarray*}
	\langle y^\ast_j(\omega),y\rangle &=& \langle y^\ast_j(\omega)-y^\ast,y\rangle + \langle y^\ast,y\rangle\\
	 &< & \delta/2 + \langle y^\ast,y\rangle  =  \delta/2 + \alpha - \delta\\
	 & < & \delta/2 + (\alpha_j(\omega)+\delta/2) - \delta = \alpha_j(\omega).
\end{eqnarray*}
Hence,  $(x,y)\in \B_{\varepsilon_j(\omega)}(x_j(\omega))\times \{u\in X \colon \langle y_j^\ast(\omega),u\rangle<\alpha_j(\omega)\}$. \qed
\newline  \noindent \vspace{0.2cm} {\bf Claim~3:}  {\em There exist $(x_k^\ast,\beta_k,z_k^\ast,\gamma_k)$ measurable functions such that 
\begin{equation*}
	\gph M_\omega = \bigcap_{k\in \N}(A_k(\omega)\times B_k(\omega))^c \textrm{ for  all  } \omega\in \Omega,
\end{equation*}
where $A_k$ and $B_k$ are defined by
\begin{equation*}
\begin{aligned}
A_k(\omega) &= \{x\in H \colon \langle x,x_k^\ast(\omega)\rangle-\frac{1}{2}\|x\|^2>\beta_k(\omega)\},\\
B_k(\omega) &= \{u\in X \colon \langle z_k^\ast(\omega),u\rangle>\gamma_k(\omega)\}.
\end{aligned}
\end{equation*}}
\noindent \emph{Proof of Claim 3}: On the one hand, by virtue of Claim 1, 
\begin{eqnarray*}
	 \B_{\varepsilon_k(\omega)}(x_k(\omega)) &=& \{x\in H \colon \frac{1}{2}\|x-x_k\|^2<\frac{1}{2}\varepsilon_k(\omega)^2\}\\
	 &=& \{ x\in H \colon \frac{1}{2}\|x\|^2-\langle x,x_k(\omega)\rangle + \frac{1}{2}\|x_k(\omega)\|^2<\frac{1}{2}\varepsilon_k(\omega)^2 \}\\
	&=& \{ x\in H \colon \langle x,x_k^\ast(\omega)\rangle-\frac{1}{2}\|x\|^2>\beta_k(\omega) \} =:A_k(\omega),
\end{eqnarray*}
where $\beta_k(\omega):= \frac{1}{2}\|x_k(\omega)\|^2-\frac{1}{2}\varepsilon_k(\omega)^2$ and $x_k^\ast(\omega) = x_k(\omega)$. On the other hand, by defining  $z_k^\ast = -y_k^\ast$ and $\gamma_k = -\alpha_k$, we can take $B_k(\omega):= \{u\in X \colon \langle z_k^\ast(\omega),u\rangle>\gamma_k(\omega)\}$, which proves the claim.

\noindent {\bf Claim~4:}  {\em Theorem \ref{MainTheorem} holds:\\}
Following the ideas from \cite{MR1920049}, let us  consider a nondecreasing  $C^\infty$ convex function $\theta : \R \to [0,+\infty)$ such that 
	\begin{align}\label{functiontheta}
	\theta(s )=\begin{cases}
	0 & \textrm{ for } s\leq 0,\\
	s+b & \textrm { for } t\geq 1,
	\end{cases}
	\end{align}
	for some $b \in (-1,0)$. It is important to emphasize that $b$ does not play any role in the proof. However, as far as we know, it is not possible to find such a function for $b =0$. \newline \noindent The function  $\theta$ satisfies the following inequality:
	\begin{align}\label{inq0001}
	\theta(s) &\leq \theta(1) + s + |b|\text{ for all } s\in \R.
	\end{align}
	Moreover, for each $k \in  \N^{\ast}$,  its  $k$-th  derivative is  uniformly bounded, i.e., 
	$$\| \theta^{(k)} \|_\infty:=\sup\{ |\theta^{(k)}(s)| \colon  s\in \R \} < \infty.$$
	Next,  we consider the function $\varphi: \Omega\times H\times X\to [0,\infty)$ defined by 
\begin{equation*}
\varphi_\omega(x,y) = \sum_{n\in \N} \frac{\theta_n^1(\omega,x)\cdot \theta_n^2(\omega,y)}{\zeta_n(\omega)^n\xi_n(\omega)^n 2^n},
\end{equation*}  
where 
\begin{equation*}
\begin{aligned}
    \theta_n^1(\omega,x) &:= \theta\left(\frac{1}{\zeta_n(\omega)}(\langle x_n^\ast(\omega),x\rangle -\frac{1}{2}\|x\|^2-\beta_n(\omega))\right),\\
    \theta_n^2(\omega,y) &:= \theta\left(\frac{1}{\xi_n(\omega)}(\langle z_n^\ast(\omega),y \rangle-\gamma_n(\omega))\right),\\
\zeta_n(\omega) &:= 1+|\beta_n(\omega)| + \|x_n^\ast(\omega)\|,\\
 \xi_n(\omega) &:= 1+\|z_n^\ast(\omega)\|+|\gamma_n(\omega)|.
	\end{aligned}
\end{equation*}
   By virtue of Claim 3, it is easy to see that  for all  $w\in \Omega$, the following equivalence holds:
   $$\varphi_\omega(x,y) = 0\iff (x,y)\in \gph M_\omega.$$
      Since $\theta$ is convex, the map $y\mapsto \varphi_\omega(x,y)$ is convex for all $(\omega,x)\in\Omega\times H$. Moreover, by \cite[proposition~4.1]{MR1016045}, the function $x \mapsto \varphi_\omega (x,y)  + L\|y\|\| x\|^2$ is convex, for some $L\geq 0$, depending of the Lipschitz constant of $\theta$.\newline 
   \noindent    Using the Faà  di Bruno's formula (see, e.g., \cite[Lemma 5.1]{MR4093728}),  for all $k\in\N^{\ast}$, there are constants $C_k, R>0$ (independent of $\omega$, $x$ and $y$) such that $$ \|D^k \theta^1_n(\omega,x)\| \leq C_k(\|x\|^k + 1) \text{ and } \|D^k \theta^2_n(\omega,y)\| \leq R.$$ Then, by the Leibniz Rule applied to $\varphi_\omega$, these inequalities imply that \eqref{eqder1} and \eqref{eqder2} hold. Thus, $\varphi_\omega$ is a $C^\infty$ function.\newline 
   \noindent  To end the proof, it remains to verify assertion $(f)$. Let  $x_k\to x$ in $H$ and $y_k\rightharpoonup y$ in $X$. It is clear that for all $n\in \N$ and $\omega\in\Omega$,  $\theta_n^1(\omega,x_k)\to \theta_n^1(\omega,x)$ and $\theta_n^2(\omega,y_k)\to \theta_n^2(\omega,y)$ as $k\to +\infty$. Hence, as $k\to +\infty$, $$ \frac{\theta_n^1(\omega,x_k)\cdot \theta_n^2(\omega,y_k)}{\zeta_n(\omega)^n\xi_n(\omega)^n 2^n} \to \frac{\theta_n^1(\omega,x)\cdot \theta_n^2(\omega,y)}{\zeta_n(\omega)^n\xi_n(\omega)^n 2^n}.$$ Finally, by the convergence dominated theorem,  we conclude that  for all $\omega \in \Omega$, $\varphi_\omega(x_k,y_k)\to  \varphi_\omega(x,y)$, which ends the proof of the Theorem. 
\end{proof}
\begin{remark} The previous theorem was stated in a Hilbert space $H$. However, a similar result can be obtained in smooth  Banach spaces if we accept a less regular function $\varphi$. Indeed, assume that $H$ is a smooth separable Banach space. Then, we can replace the set $A_k(\omega)$  of Claim 3 by $$A_k(\omega) = \{ x\in H\colon\|x-x_k(\omega)\|^p<\varepsilon_k(\omega)^p \},$$ where $p>1$ is fixed. In Claim 4, the functions $\theta_n^1$ and $\zeta_n$ can be modified by 
\begin{equation*}
\begin{aligned}
\theta_n^1(\omega,x)&= \theta(\varepsilon_n(\omega)^p-\|x-x_n(\omega)\|^p),\\
\zeta_n(\omega)&= 1+\varepsilon_n(\omega)^p+\|x_n(\omega)\|^{p-1}.
\end{aligned}
\end{equation*} 
The obtained function $\varphi$ will be $C^1$ and we can find a constant $C>0$ such that for all $(\omega,x,y)\in \Omega\times H\times X$
    $$  \|D\varphi_\omega(x,y)\|\leq C(\|x\|^{p-1}+1)(\|y\|+1). $$ Furthermore,  \eqref{eqder2} still holds true and the function $y\mapsto \varphi_\omega(x,y)$ is  $C^\infty$. \qed
\end{remark}

As a consequence of Theorem \ref{MainTheorem}, we obtain a generalization of the main result in \cite{MR1920049}. Indeed, \cite[Theorem~1]{MR1920049} establishes that every closed convex set in a separable Banach space can be seen as the set of minimizers of a $C^\infty$ convex function. Here, we obtain a stronger conclusion: the values of any measurable multifunction with closed and convex values can be written as the minimizers of a $C^\infty$ convex normal integrand.
\begin{corollary}\label{mainconvex}  
		Let $M: \Omega \tto X$ be measurable multifunction with closed and convex values. Then, there exists $\varphi \colon \Omega \times X\to [ 0, +\infty)$ a convex normal integrand     such that  for all $\omega\in \Omega$ the map $x\mapsto \varphi_{\omega}(x)$ is $C^{\infty}$ and
		\begin{equation}\label{eqmainconvex}
		M(\omega) =\{ x \in X\colon  \varphi_{\omega}(x) = 0 \} \textrm{ for all }\omega \in \Omega.
		\end{equation}
		Moreover, for all $k\in\N^{\ast}$,
		\begin{align}\label{eq00main}
		\sup_{(\omega,x)\in \Omega\times X} \| D^k \varphi_\omega ( x)\|< +\infty.
		\end{align} 
	\end{corollary}
	
\begin{proof} Define $H=\{0\}$.  It is enough to apply Theorem \ref{MainTheorem} to the set-valued map $\hat{M}:\Omega\times H\tto X$ defined by $\hat{M}(\omega,x) = M(\omega)$. Indeed, it is clear that $\gph \hat{M}_\omega = H\times M(\omega)$ and, thus, $\omega\tto \gph \hat{M}_\omega$ is a random multifunction (i.e.,  is measurable with closed values). Moreover, for all $\omega\in \Omega$, $\hat{M}_\omega$ is pseudo-norm-weak upper semicontinuous. Hence, by virtue of Theorem \ref{MainTheorem}, we can find $\phi:\Omega\times \{0\}\times X\to [0,+\infty)$ a  $C^\infty$ normal integrand. Then, the function $\varphi_\omega:=\phi_\omega(0,\cdot )$ is  a $C^\infty$ normal integrand over $\Omega\times X$. Moreover, we have that for a.e. $\omega\in \Omega$,
    \begin{equation*}
y\in M(\omega)\iff (0,y)\in \gph \hat{M}_\omega \iff\phi_\omega(0,y) = 0 \iff \varphi_\omega(y) = 0.
    \end{equation*}
    The convexity of $\varphi_\omega$ follows from Theorem \ref{MainTheorem_c}. Finally,  \eqref{eq00main} follows from Theorem \ref{MainTheorem_d}. 
\end{proof}

Theorem \ref{MainTheorem}  allows us to provide also a representation for measurable multifunctions with merely closed values on separable Hilbert spaces.
\begin{corollary}\label{mainclosed}
	Let $H$ be a separable Hilbert space and $M: \Omega \tto H$ be a measurable multifunction with closed values. Then, there exists a normal integrand function $\varphi :\Omega \times H\to [0, +\infty)$ such that 
	\begin{enumerate}
	\item[(a)] for all  $\omega\in \Omega$, $M(\omega)=\{x\in H \colon \varphi_{\omega}(x)=0\}$.
	\item[(b)] for all $\omega\in \Omega$ the map $x\mapsto \varphi_{\omega}(x)$ is $C^\infty$ and for some $L\geq 0$ the map $x\mapsto \varphi_\omega(x) + L\|x\|^2$ is convex
	\item[(c)] for all $k \in  \mathbb{N}^{\ast}$, there are constants $a_k, b_k >0$ such that  for all $x\in H$   
	\begin{align}\label{eq00main:closed}
	\sup_{\omega\in \Omega} \|  D^k \varphi_{\omega}(x)\| \leq a_k  \| x\|^k + b_k.
	\end{align}
	\end{enumerate}
\end{corollary}
\begin{remark}
It is worth emphasizing that in Corollary \ref{mainclosed}, although the values of $M(\omega)$ are merely closed, the obtained function $\varphi_{\omega}$ is not so far from being a convex one. Indeed, $\varphi_{\omega}$ is the difference of convex functions:  $\varphi_{\omega}=(\varphi_{\omega}+L\Vert \cdot\Vert^2)-L\Vert \cdot\Vert^2$.
\end{remark}
\begin{proof} Set $X=\{0\}$ and consider the multifunction $\hat{M}:\Omega\times H\tto X$ defined by
     \[\hat{M}(\omega,x) = \begin{cases}
        \{0\} & \text{if } x\in M(\omega),\\
        \emptyset & \text{if } x\notin M(\omega).
    \end{cases}\]
    It is clear that $\gph \hat{M}_\omega = M(\omega)\times \{0\}$. Thus, $\hat{M}$ is a random multifunction. To prove that $\hat{M}$ satisfies  the pseudo-norm-weak usc property,  fix $\omega \in \Omega$ and assume that $\hat{M}_\omega(x)\subset \{ u\in X \colon  \langle y^\ast,u\rangle<\alpha \}$ for $y^\ast\in X^\ast=\{0\}$ and $\alpha\in \R$. On the one hand, if   $\hat{M}_\omega(x)=\{0\}$, then clearly $\{ u\in X \colon \langle y^\ast,u\rangle<\alpha \} =X$. On the other hand, if  $\hat{M}_\omega(x) = \emptyset$, then $x\notin M(\omega)$. Since $M(\omega)$ is closed,  there exists $\varepsilon >0$ such that $\B_\varepsilon(x)\cap M(\omega)=\emptyset$ and then $\hat{M}_\omega(x') = \emptyset$ if $x'\in \B_\varepsilon(x)$. Hence, for every $\eta>0$,  $\hat{M}_\omega(x')\subset\{ u\in X \colon \langle y^\ast,u\rangle<\alpha+\eta \}$.  
    Therefore, in any case, the pseudo-norm-weak usc property holds. \newline 
    By virtue of  Theorem \ref{MainTheorem}, there exists $\phi:\Omega\times H\times \{0\}\to [0,+\infty)$  a $C^\infty$ normal integrand representing the set-valued mapping $\hat{M}$. Thus, the function $\varphi_\omega := \phi_\omega(\cdot,0)$ is a $C^\infty$ normal integrand such that for a.e. $\omega\in \Omega$
 \begin{equation*}
         x\in M(\omega)  \iff  (x,0)\in \gph \hat{M}_\omega\iff \phi_\omega(x,0) = 0 \iff \varphi_\omega(x) = 0.
    \end{equation*}
    Finally,  due to Theorem \ref{MainTheorem}, the map $x\mapsto \varphi_\omega(x)+L\|x\|^2$ is  convex and for some $L\geq 0$  the inequality (\ref{eq00main:closed}) holds.
\end{proof}

We end this section by providing a representation result for multifunctions with values in dual spaces, similar to Corolary \ref{mainconvex}, where we prove that lower semicontinuity of the epigraphs of the support functions induces continuity in both variables for normal integrands.
\begin{theorem} Let $T$ be a metric space and  $X$ a separable Banach space. Assume that $C\colon T\tto X^\ast $ is a multifunction with nonempty, $w^\ast$-closed and  convex  values so that $t \tto \epi \sigma_{C(t)}$ is lsc. Then, there exists a $C^\infty$-convex  normal integrand $\varphi :T \times X^\ast \to [ 0, +\infty)$  such that 
$$x^\ast\in C(t) \Leftrightarrow \varphi(t,x^\ast)=0 \textrm{ for all } (t,x^\ast)\in T\times X^\ast.$$
Moreover, $\varphi$ can be chosen so that, for all $k \in  \mathbb{N}^{\ast}$, the map  $(t,x^\ast)  \to  D^k \varphi_t( x^\ast)$ be continuous  with $\displaystyle\sup_{(t,x^{\ast})\in T\times X^{\ast}} \| D^k \varphi_t(x^\ast)\| < +\infty$.
\end{theorem}
\begin{proof} As a consequence of Michael's selection theorem (see, e.g., \cite[Theorem 6.3.11, p. 491]{MR2527754}), there are continuous functions $z_n \colon T\to X^{\ast}$ and $\gamma_n \colon T \to \R$ such that 
\begin{equation*}
	\epi \sigma_{C(t)}= \cl\{ (z_n(t),\gamma_n(t))_n \} .
\end{equation*}
	Then, we consider the function $\varphi: T \times X^\ast\to [0,\infty)$ defined by 
\begin{equation*}
\varphi_t(x^\ast) = \sum_{n\in \N} \frac{1}{2^n}\frac{ \theta_n(t,x^\ast)}{\xi_n(t)},
\end{equation*}  
where  
\begin{equation*}
 \theta_n(t,x^\ast) := \theta\left(\frac{1}{\xi_n(t)}(\langle x^\ast,z_n(t) \rangle-\gamma_n(t))\right) \textrm{ and } \xi_n(t)  :=  1+\|z_n(t)\|+|\gamma_n(t)|.
\end{equation*} 
Here $\theta$ is the smooth convex function defined in \eqref{functiontheta}. Hence, the result follows from similar arguments to the given in proof of  Theorem \ref{mainconvex}.
\end{proof}

\section{Approximation of random multifunctions and functions}\label{Approxima}

In this section, by using Theorem \ref{MainTheorem}, we provide smooth  approximations to random multifunctions and normal integrands.

\begin{corollary}\label{CorAprRandomMult}
Let $H$ be a separable Hilbert space and $X$ a separable Banach space. Let $M\colon\Omega \times H \tto X$  be a random multifunction such that,  for all $\omega\in \Omega$, $M_\omega$ is pseudo-norm-weak upper semicontinuous with convex values.  Then, there exists a sequence of random multifunctions $ M^k\colon\Omega \times  H \tto X $ with convex values such that:
\begin{enumerate}[label=\alph*)]
    \item \label{CorAprRandomMult_a} For all $\omega\in \Omega$, $\gph M_\omega ^{k+1}\subset \gph M_\omega ^k$ for all $k\in \N$ and  $$\bigcap_{k\in \N } \gph M_\omega ^k = \gph M_\omega.$$ 
    \item \label{CorAprRandomMult_b} For all $\omega\in \Omega$, $\| \cdot\|\times w$-$\limsup \gph M_\omega ^k \subset  \gph M_\omega$. 
    \item \label{CorAprRandomMult_c} For all $\omega\in\Omega$, if $x\in \operatorname{dom}(M_\omega)$ and $k\in \N$, there exists a neighbourhood $U$ of $x$ such that for all $x'\in U $ the sets $M_\omega^k(x')$ are $C^\infty$-convex bodies.
    \item \label{CorAprRandomMult_d} If $\operatorname{dom}(M_\omega)=H$ for a.e. $\omega\in \Omega$, then  for all $k\in\N$ the set $\gph M_\omega^k$ has  $C^\infty$-smooth boundary for a.e. $\omega\in \Omega$.
\end{enumerate}
\end{corollary} 
\begin{proof}
By Theorem \ref{MainTheorem}, there is a $C^\infty$ function $\varphi$ such that  for all $\omega \in \Omega$ 
$$ (x,y)\in \gph(M_\omega)\iff \varphi_\omega(x,y)=0.$$ 
Consider $\varepsilon_k \to  0^+$ and the multifunction $M^k (\omega,x) := \{  z\in X: \varphi_\omega(x,z) \leq \varepsilon_k\}$. 
We will prove that $(M^{k})_{k\in \mathbb{N}}$ is the required sequence of multifunctions. \vspace{0.1cm}\newline  
\noindent {\bf Claim~1:} \emph{$M^{k}$ is a random multifunction}\newline
\noindent \emph{Proof of Claim 1}: Due to the continuity of $\varphi_\omega$ and the measurability of $\varphi(\cdot,x,y)$ for all $(x,y)\in H\times X$, it is clear that  $M^k$ is a random multifunction with convex values. Hence, by virtue of the continuity of $\varphi_\omega$, the multifunction $\omega\tto \gph(M^k_\omega)$ is measurable with closed values. Besides, by the convexity  and continuity of  $\varphi_\omega(x,\cdot)$, the set $M^k(\omega,x)$ is closed and convex. \vspace{0.1cm}\newline 
\noindent {\bf Claim~2:} \emph{Assertion $a)$ holds.}\newline
\noindent \emph{Proof of Claim 2}: On the one hand, it is clear that $\gph(M_\omega)\subset \gph(M_\omega^k)$ for all $k\in\N$.  On the other hand, for any $(x,y)\in H\times X$ such that  $\varphi_\omega(x,y)\leq \varepsilon_k$ for all $k\in \mathbb{N}$ we have that $\varphi_\omega(x,y) = 0$. Indeed, since $\varepsilon_k\to 0^+$, we obtain that $\bigcap_{k\in \N } \gph M_\omega ^k = \gph M_\omega$ for all $\omega \in \Omega$. \vspace{0.1cm}\newline
\noindent {\bf Claim~3:} \emph{Assertion $b)$ holds}\newline
\noindent \emph{Proof of Claim 3}: Let $(x,y)\in \|\cdot \|\times w$-$\limsup \gph M_{\omega}^k$. Then, there exists a sequence $(x_{n_k},y_{n_k})\in \gph(M_\omega^{n_k})$, where $(n_k)_{k\in \N}$ is a strictly increasing, such that $x_{n_k}\to x$ and $y_{n_k}\rightharpoonup y$. Hence, since $\varphi$ is a $\|\cdot\|$-weak continuous and $\varphi_\omega(x_{n_k},y_{n_k})\leq \varepsilon_{n_k}$, it follows that $\varphi_\omega(x,y) = 0$. Thus,  $(x,y)\in \gph M_\omega$ and $$\|\cdot \|\times w\text{-}\limsup \gph M_{\omega}^k \subset \gph M_\omega \textrm{ for all }   \omega \in \Omega.$$  
\noindent {\bf Claim~4:} \emph{Assertion $c)$ holds} \newline
\noindent \emph{Proof of Claim 4}: Fix $x\in \text{dom}(M_\omega)$ and $k\in \N$. Then,  there exists $y\in X$ such that $\varphi_\omega(x,y)=0$. By continuity of $\varphi_\omega$, it is possible to find a neighborhood $U\times V$ of $(x,y)$ such that the following implication hold:
$$(x',y')\in U\times V \Rightarrow \varphi_\omega(x',y')<\varepsilon_k.$$ 
Let $x'\in U$. Then, since $V\subset M_\omega^k(x')$,  $M_\omega^k(x')$ is a closed and convex set with nonempty interior. 
Let $z\in \bd M_\omega^k(x') := \{ v\in X:\varphi_\omega(x',v)=\varepsilon_k \}$. Then, by observing that $\varphi_\omega(x',\cdot)$ is a $C^\infty$ and convex function, we obtain that
$$
D_y\varphi_\omega(x',v) = 0 \iff v \textrm{ is a minimum of } \varphi_{\omega}(x',\cdot).
$$
Moreover, $\varphi_\omega(x',y)<\varepsilon_k=\varphi_\omega(x',z)$. Thus,  $D_y\varphi(x',z)\neq 0$, which, by the implicit function theorem, implies that $\bd M_\omega^k(x')$ is a $C^\infty$-manifold. \vspace{0.1cm}
\newline \noindent {\bf Claim~5:} \emph{Assertion $d)$ holds}: \newline
\noindent \emph{Proof of Claim 4}: Suppose that $\text{dom}(M_\omega) = H$. By assumption, it is clear that  the set $\{ (x,y)\in H\times X\colon \varphi_\omega(x,y)<\varepsilon_k \}$ is nonempty. Then,  $\gph M_\omega^k$ has nonempty interior. Moreover,  
$$\operatorname{bd} \gph M_\omega^k\subset \{ (x,y)\in H\times X\colon\varphi_\omega(x,y) = \varepsilon_k \},$$
which implies that any $(x,y)\in \operatorname{bd} \gph M_\omega^k$  satisfies $\varphi_\omega(x,y) = \varepsilon_k$. Even more, there exists $y'\in X$ such that $\varphi_\omega(x,y') = 0$. Hence, if $D \varphi(x,y) = 0$, then $D_y\varphi(x,y) = 0$. Thus,  by convexity,  $y$ is a minimum of function $\varphi_\omega(x,\cdot)$. However, since $\varphi_\omega(x,y')<\varphi_\omega(x,y)$, we obtain a contradiction.  Therefore, $D\varphi_\omega(x,y)\neq 0$ and we can apply the implicit multifunction theorem to obtain that $\operatorname{bd} M_\omega^k(x')$ is a $C^\infty$-manifold.
\end{proof}

The latter result provides a new approximation technique for random multifunctions. Indeed, in the next result, we construct a sequence of approximations whose values are smooth convex bodies. Recall that the notion of essentially uniformly convergence of multifunctions was defined in  \eqref{metric1}.
\begin{corollary}\label{corconvexbody}
	Let $M:\Omega \tto \mathbb{R}^d $ be a measurable multifunction with nonempty closed and convex values. Then, there exists a sequence of measurable multifunction $S_k: \Omega\tto\mathbb{R}^d$ converging essentially uniformly to $M$ and whose values are  $C^\infty$ convex bodies for all  $\omega\in \Omega$.
\end{corollary}
\begin{proof}
Set $X = \R^d$ and $H = \{0\}$ and consider the sequence of random multifunctions $(M_k)_{k}$ provided by Corollary \ref{CorAprRandomMult}. Then, $M_k$ takes $C^\infty$ convex bodies values for a.e. $\omega\in \Omega$ and by assertion $a)$ from Corollary \ref{CorAprRandomMult}, for a.e. $ \omega\in \Omega$, 
\begin{equation}\label{conv-distancia}
\bar{d}(M_{\varepsilon_k}(\omega),M(\omega)) \to 0 \textrm{ as } k\to +\infty.
\end{equation}
where $\bar{d}$ is the integrated set distance defined in \eqref{metric0}. Let us consider the function  $f\colon \Omega \times \N \to [0,+\infty)$ defined by $f(\omega,k)= \bar{d}(M_{\varepsilon_k}(\omega),M(\omega))$. It is clear that  $\omega \to f(\omega ,k)$ is measurable for all $k\in \N$. Thus,  for all $n\in\mathbb{N}$, the multifunction $J_n: \omega\rightrightarrows \{ k\in\N\colon f(\omega,k)\leq 1/n \} $ is measurable. Moreover,   due to \eqref{conv-distancia},  $J_n(\omega)$ is nonempty. Hence, there exists a measurable selection $\lambda_n:\Omega\to \N$ such that
\begin{equation*}
 \bar{d}(M_{\varepsilon_{\lambda_n(\omega)}}(\omega),M(\omega))\leq \frac{1}{n} \textrm{ for all } \omega\in \Omega.
\end{equation*}
Finally, by virtue of \eqref{conv-distancia}, the measurable multifunction $S_n\colon \omega\rightrightarrows M_{\varepsilon_{\lambda_n(\omega)}}(\omega)$ converges essentially uniformly to $M$ with respect to set distance.
 \end{proof}

The next result concerns the approximation of normal integrands by $C^{\infty}$ normal integrands.
\begin{theorem}\label{thmfuncap}
Let $f\colon\Omega \times H \times X \to {\R}$ be a normal integrand. Assume that $f$ is convex with respect to the variable on $X$ and that the conjugate map 
\begin{equation}\label{eqconjfenchel}
    x\mapsto f_\omega^\ast(x ,x^\ast):=\sup_{y\in X}\langle x^\ast,y\rangle-f_\omega(x,y)    
\end{equation}
 is upper semicontinuous for all $\omega\in \Omega$ and $x^\ast\in X^\ast$. Then, there exists a nondecreasing sequence of $C^{\infty}$-normal integrand $f_k:\Omega \times H \times X \to \R$ such that  for all $\omega \in \Omega$
 \begin{equation}\label{eq_supremum}
 \lim\limits_k f_k(\omega,x,z)=\sup\limits_{k}  f_k(\omega,x,z)=f(\omega, x,z) \quad \textrm{ for all } (x,z)\in H\times X.
\end{equation} 
 Moreover, the function $f_k$ are convex with respect to the variable on $X$.
\end{theorem}
\begin{proof} Let us consider the multifunction $M\colon \Omega \times H \to X\times \R$ defined by
 \begin{equation}\label{defM}
     M_\omega(x):=\{ (y, \alpha)\in X\times \R : f_\omega (x,y) \leq \alpha\}.
  \end{equation}
  We observe that $\gph M_{\omega}=\epi f_{\omega}$ for all $\omega \in \Omega$. \newline 
 \noindent {\bf Claim~1:}  {\em The application $M$ is a random multifunction with closed and convex values. Moreover, for every $\omega \in \Omega$, $x\tto M_\omega(x)$ is pseudo-norm-weak upper semicontinuous.} \newline
\noindent \emph{Proof of Claim 1}:  Since $f$ is a normal integrand, $M$ is a random multifunction with nonempty values.  The convexity of the values of $M_{\omega}(x)$ follows from the convexity of $f$ with respect to the variable on $X$. \newline 
  \noindent To prove the pseudo-norm-weak usc property, take $w^\ast\in X^\ast$ and $\alpha,\gamma\in\R$ such that $M_\omega(x) \subset C_{w^\ast,\alpha,\gamma}:= \{  (u,\beta)\in X\times \R\colon \langle w^\ast , u\rangle + \alpha \beta < \gamma\}$. By the definition of $C_{w^\ast,\alpha,\gamma}$, it follows that $\alpha \leq 0$. \newline \noindent On the one hand, if $\alpha =0$, then  $w^\ast = 0$ and $\gamma >0$. Hence,  $C_{w^\ast,\alpha,\gamma} = X\times \R$. \newline \noindent On the other hand, if $\alpha <0$, then
$$ \langle |\alpha|^{-1}w^\ast, u \rangle - f_\omega(x,u) <\gamma|\alpha|^{-1} \textrm{ for all } u \in X,$$ which implies that $f_\omega^\ast(x,|\alpha|^{-1}w^\ast)\leq \gamma|\alpha|^{-1} $. By hypothesis of upper semicontinuity of the conjugate map \eqref{eqconjfenchel}, for all $\eta >0$, there exists $ \varepsilon >0$ such that $f^\ast(x',|\alpha|^{-1}w^\ast) < (\gamma+\eta)|\alpha|^{-1}  $ for all $x'\in \mathbb{B}_\varepsilon(x)$.  Therefore,  for all $u\in X$  $$ \langle |\alpha|^{-1}w^\ast, u \rangle - f_\omega(x',u) <(\gamma+\eta)|\alpha|^{-1}. $$ 
\noindent Finally, if $(u,\beta)\in M_\omega(x')$ for $x'\in \mathbb{B}_\varepsilon(x)$, then $(u,\beta)\in C_{w^\ast,\alpha,\gamma+\eta}$, which proves that $x\in H\tto M_\omega(x)$ is pseudo-norm-weak upper semicontinuous for all $\omega\in\Omega$.  \vspace{0.1cm} \newline
\noindent {\bf Claim~2:} {\em There exists a sequence of normal integrands $(f^k)_k$ satisfying the statement of theorem.} \newline
\noindent \emph{Proof of Claim 2}: Since $\dom M_\omega = H$ for all $\omega\in \Omega$, we can apply Corollary \ref{CorAprRandomMult} to obtain a sequence of random multifunctions $(M^k)_k$ such that for all $k\in \N$ and a.e $\omega\in \Omega$, the set $\gph M^k_\omega$ has $C^\infty$-smooth boundary and 
  $$\bigcap_{k\in \N } \gph M_\omega ^k = \gph M_\omega \textrm{ for a.e. } \omega \in \Omega.$$
Thus, for all $k\in\N$, there exists a $C^\infty$-normal integrand $f^k \colon \Omega\times H\times X\to\R$ such that $\gph M_w^k = \epi f_\omega^k $ for a.e. $\omega\in \Omega$ and all $k\in \N$. Moreover, for a.e. $\omega\in \Omega$, $f_\omega^{k+1}\leq f_\omega^k\leq f_\omega$ pointwisely. Thus, 
$$\sup_{k\in\N}f_\omega^k(x,y)\leq f_\omega(x,y) \textrm{ for all } (x,y)\in H\times X.$$
Hence, if $(x,y)\in H\times X$, we can take $\alpha:=\sup_{k\in\N}f_\omega^k(x,y)$, which satisfies $(x,y,\alpha)\in \gph M_\omega^k$ for all $k\in\N$. Finally, $(x,y,\alpha)\in \gph M_\omega$, then $f(x,y)\leq \alpha$, which proves \eqref{eq_supremum}.
\end{proof}
The following result gathers some sufficient conditions which allow us to verify the condition \eqref{eqconjfenchel} from Theorem \ref{thmfuncap}.
\begin{proposition} Under the assumptions of Theorem \ref{thmfuncap}, the upper semicontinuity condition \eqref{eqconjfenchel} can be verified in the following cases:
\begin{enumerate}
\item[(a)] For all $\omega\in \Omega$, the function  $f_\omega:H\times X\to \R$ is uniformly continuous.
\item[(b)] The space $X$ is reflexive, for all $\omega \in \Omega$, if $x_n \to x$ and $y_n\rightharpoonup y$, then
$$f_{\omega}(x,y)\leq \liminf  f_{\omega}(x_n,y_n),$$
and  for all $\omega \in \Omega$ and $x\in H$, there exist $\delta>0$, $\alpha>0$ and $\beta\in \R$ such that
 \begin{equation}\label{eqnprop}
       \alpha\|y\|^2+\beta \leq  f_\omega(x',y)   \textrm{ for all } (x',y)\in \B_\delta(x)\times X .
    \end{equation}
\end{enumerate}

\end{proposition} 
\begin{proof} 
$(a)$: Fix $\omega\in \Omega$ and let $y^\ast\in X^\ast$, $\varepsilon >0$ and a sequence $(x_n)_n\subset H$ converging to $x$. 
By virtue of the uniform continuity of $f_{\omega}$, there exists $\delta>0$ such that if $\|x-x'\|<\delta$ and  $\|y-y'\|<\delta$ then $|f(x,y)-f(x',y')|<\varepsilon$. Then, since $x_n\to x$, there exists $N\in\N$ such that $\|x_n-x\|<\delta$ for $n\geq N$. Thus, for all $y\in X$ and $n\geq N$, $|f_{\omega}(x_n,y)-f_{\omega}(x,y)|<\varepsilon$.  Hence, for $n\geq N$, $$ \sup_{y\in X} f_{\omega}(x,y)-f_{\omega}(x_n,y)\leq \varepsilon.
$$ 
Therefore, by using the latter inequality, we obtain that for $n\geq N$,
    \begin{equation*}
    \begin{aligned}
         f^\ast_\omega(x_n,y^\ast) &= \sup_{y\in X} \langle y^\ast,y\rangle - f_\omega(x_n,y)\\
        &= \sup_{y\in X} \langle y^\ast,y\rangle -f_\omega(x,y)+f_\omega(x,y)-f_\omega(x_n,y)\\
     &\leq \sup_{y\in X} \langle y^\ast,y\rangle -f_\omega(x,y) + \sup_{y\in X}f_\omega(x,y)-f_\omega(x_n,y)\\
      &\leq  f^\ast_\omega(x,y^\ast) + \varepsilon.
        \end{aligned}
    \end{equation*}
    By taking $n\to\infty$ in the above inequality,  we obtain that $$ \limsup_{n\to\infty} f^\ast_\omega(x_n,y^\ast)\leq f^\ast_\omega(x,y^\ast) + \varepsilon,$$
    which implies \eqref{eqconjfenchel}. \newline \noindent 
$(b)$ Fix $y^\ast\in X^\ast$ and $\omega\in \Omega$. We will prove first that the map $x\mapsto f_\omega^\ast(x,y^\ast)$ takes finite values. Indeed, by virtue of \eqref{eqnprop}, for all $y\in X$
    \begin{equation*}
    \begin{aligned}
       \langle  y^\ast,y\rangle-f_\omega(x,y)&\leq \|y^\ast\|\|y\|-\alpha\|y\|^2-\beta\\
      &= -\alpha\left(\|y\|-\frac{\|y^\ast\|}{2\alpha}\right)^2 - \beta - \frac{\|y^\ast\|^2}{4\alpha^2}\\
      &\leq - \beta - \frac{\|y^\ast\|^2}{4\alpha^2},
\end{aligned}
    \end{equation*}
    which implies that $f^\ast_\omega(x,y^\ast)\leq - \beta - \frac{\|y^\ast\|^2}{4\alpha^2}<+\infty$ for all $x\in H$. \newline \noindent 
    To prove the upper semicontinuity, we proceed by contradiction. Suppose that the map $x\mapsto f_\omega^\ast (x,y^\ast)$ is not upper semicontinuous at $x\in H$. Then, there exist $(x_n)\subset H$ converging to $x$ and $\varepsilon>0$ such that
    \begin{equation*}
    f_\omega^\ast(x,y^\ast)+\varepsilon\leq   \limsup_{n\to\infty}f_\omega^\ast(x_n,y^\ast).
    \end{equation*}
    Thus, up to a subsequence,  $$ f_\omega^\ast(x,y^\ast)+\frac{1}{2}\varepsilon\leq f_\omega^\ast(x_{n},y^\ast) \textrm{ for } n   \textrm{ big enough}.$$
     Let $\delta>0$ such that \eqref{eqnprop} holds. Then, there exists $N\in \N$ such that $x_{n}\in \B_\delta(x)$ for all $n\geq N$. Using the definition of the convex conjugate, we can find  $(y_{n})_{n}\subset X$ such that for all $n$ big enough
     \begin{equation*}
      f_\omega^\ast(x_{n},y^\ast)\leq \frac{\varepsilon}{4} + \langle y^\ast,y_{n}\rangle-f_\omega(x_{n},y_{n}).
    \end{equation*}
    Therefore, for $n$ large enough
    \begin{equation}\label{eqnineq1} 
    f_\omega^\ast(x,y^\ast)+f_\omega(x_{n},y_{n})+\frac{\varepsilon}{4} \leq \langle y^\ast,y_{n}\rangle.
    \end{equation}
    Moreover, since for all $n$ large enough
    $$\langle y^\ast,y_{n}\rangle-f_\omega(x_{n},y_{n})\leq -\alpha\left(\|y_{n}\|-\frac{\|y^\ast\|}{2\alpha}\right)^2 - \beta - \frac{\|y^\ast\|^2}{4\alpha^2},$$
    the sequence $(y_{n})$ is bounded. Thus, without loss of generality, we can assume that $y_{n}$ weakly converges to some $y\in X$. Then, by assumption,  $ f_\omega(x,y)\leq \liminf f_\omega(x_{n},y_{n})$ and $\langle y^\ast,y_{n}\rangle\to \langle y^\ast,y\rangle$. Finally, by taking limit in \eqref{eqnineq1}, we obtain a contradiction with the definition of convex conjugate.
    \end{proof}

To prove our next result, we need to extend the notion of prox-bounded function to the framework of normal integrands (see, e.g., \cite[Definition 1.23]{MR1491362}).  
A normal integrand  $f\colon \Omega \times  \R^d \to \Rex$ is said to be  \emph{prox-bounded} if there exists a measurable function $\lambda : \Omega \to (0,+\infty) $ such that for all $\omega \in \Omega$  there exists $x\in \R^d$  satisfying $$e_{\lambda(\omega)}f_\omega (x):=\inf_{z\in \R^d}\{ f_\omega(z) + \frac{1}{2\lambda(\omega)  }  \| x-z\|^2\}>-\infty.$$

Our next result is a functional counterpart of Corollary \ref{corconvexbody} and  can also be seen as an  extension of Theorem  \ref{thmfuncap} for extended real-valued functions defined on finite-dimensional spaces. 

\begin{corollary}
	Let $f:\Omega \times \mathbb{R}^d \to \Rex$ be a prox-bounded normal integrand. Then, there exists a sequence $(f^k)_k$ of $C^\infty$-normal integrands $f_k \leq f$   converging epigraphically essentially uniformly to $f$.
\end{corollary}
\begin{proof}
By prox-boundedness of $f$ and \cite[Theorem~1.25]{MR1491362}, there exists a function $\lambda_f : \Omega \to (0,+\infty) $ such that for a.e.  $\omega \in \Omega$ and all $\lambda \in (0, \lambda_f(\omega))$ the function 
  \begin{equation*}
     \MoreauYosida{f_\omega}{\lambda }(x):=\inf_{y\in \R^d}    f_\omega (y) + \frac{1}{ 2\lambda} \| x- y\|^2,
  \end{equation*}
  is finite valued and continuous.  Thus, for all $n\in\N$, the multifunction $$M_n: \omega\tto \{ \lambda>0: \bar{d}(\epi f_\omega,\epi \MoreauYosida{f_\omega}{\lambda })\leq 1/n\},$$  is measurable. Then, by virtue of Proposition \ref{graph:measurable:selection:theorem},  we can find a measurable function $\lambda_n :\Omega \to\R $ with $\lambda_n (\omega) \in (0,\lambda_f(\omega))$ and such that 
 $$ \bar{d}( \epi f_\omega,\epi \MoreauYosida{f_\omega}{\lambda_n(\omega) }) \leq 1/n \textrm{ for all } \omega \in \Omega.$$
 Define $g^n_\omega := \MoreauYosida{f_\omega}{\lambda_n(\omega)}$. By virtue of Theorem \ref{thmfuncap} (applied on $H=\R^d$ and  $X=\{0\}$), we can find a sequence $(g_k^n)_{k}$ of $C^\infty$-normal integrands such that $g^n_{k}(\omega,\cdot)\leq g^n_{k+1}(\omega,\cdot)\leq g^n(\omega,\cdot)$ for a.e. $\omega \in \Omega$ and 
 \begin{equation*}
 \sup\limits_{k\in \N} g_k^n(\omega,x) =g^n(\omega, x) \textrm{ for all } x\in H.
\end{equation*}
 By similar arguments to the given in the proof of Corollary \ref{corconvexbody}, it is possible to find $\hat{g}^n$ such that 
 $$ \bar{d}(\epi \hat{g}^n(\omega,\cdot) ,\epi g^n(\omega,\cdot))\leq 1/n  \textrm{ for a.e } \omega \in \Omega.$$
 Thus, if we set $f^n_\omega := \hat{g}^n(\omega,\cdot)$, then for a.e. $\omega \in \Omega$
 \begin{equation*}
   \bar{d}(\epi f^n_{\omega},\epi f_{\omega}) \leq \bar{d}(\epi f^n_{\omega},\epi g^n_{\omega}) + \bar{d}(\epi g^n_{\omega},\epi f_{\omega})\leq   \frac{1}{n}+\frac{1}{n} = \frac{2}{n},
 \end{equation*}
 which implies that $(f^n)_n$ converges epigraphically essentially uniform to $f$.
 Finally, $(f^{n})_n$ is the desired sequence.
	\end{proof}


\section{Integrable selections as minimizers of integral funcionals}\label{integrableselect}
In this section, we apply Theorem \ref{MainTheorem} to show that the sets of $p$-integrable  selections of measurable multifunctions  can be represented  as the set of minimizers of a  $C^\infty$-convex integral functional. 

The first result of this section is devoted to the case of measurable multifunction with nonempty,  closed and convex values. 
\begin{theorem}\label{mainconvex2} Assume that $(\Omega, \A,\mu)$ is a finite measure space. 
	Let $M: \Omega \tto X$ be a measurable multifunction with nonempty, closed and convex values such that the set of $p$-integrable selection $S_M^p$ is nonempty. Then, there exists a $C^\infty$-convex  normal integrand $\varphi \colon \Omega \times X\to [ 0, +\infty)$  such that 
	\begin{equation}\label{eqmainconvex2}
	S_M^p = \left\{  x\in L^p(\Omega ,X) \colon \IntfLp{\varphi}(x):=\int_{\omega}\varphi_{\omega}(x(\omega))\dmu=0    \right\},
	\end{equation}
	Moreover, for $p\in (1,+\infty]$, the integral functional $\IntfLp{\varphi}$ is $C^\infty$ with $k$-derivative 
	\begin{equation}\label{derivativeconvex2}
	D^k\IntfLp{\varphi}(x)= \int_\Omega D^k \varphi_\omega( x(\omega))\dmu \textrm{ for } x\in L^p(\Omega,X),
	\end{equation}
	where the integral in the right-hand side of \eqref{derivativeconvex2} is in the sense of Gelfand.
\end{theorem}
\begin{proof}
Consider  the integrand function $\varphi$ associated to $M$ given by Theorem  \ref{mainconvex}. Let $x_0 \in S_M^p$. Then, by virtue of \eqref{eq00main}, 
$$ 0\leq \IntfLp{\varphi}(x) \leq 	\kappa   \| x -x_0\|_p, \text{ for all }x\in L^p (\Omega,X), $$
 where $\kappa:= \mu(\Omega)^{\frac{p}{p-1}}\cdot\sup_{(\omega,x)\in \Omega \times X}  \| D\varphi_\omega( x)\|$. Hence, the functional $\IntfLp{\varphi}$ is finite over  $L^p (\Omega,X)$. Equality \eqref{eqmainconvex2} can be easily verified from the properties of $\varphi$.\newline
 Next, we proceed to prove that $\IntfLp{\varphi}$ is $C^\infty$ and satisfies \eqref{derivativeconvex2}.  Indeed, by  the measurability of   $\varphi$, the integral in the right-hand side of \eqref{derivativeconvex2} is well-defined. Moreover, by Taylor's formula, for all $(\omega,x)\in \Omega \times X$ 
  \begin{align}\label{taylor}
  \| D^k\varphi_\omega( x+h) - D^k\varphi_\omega(x) - D^{k+1}\varphi_\omega(x) (h) \| \leq C_{k+2}\| h\|^2,
  \end{align}
  where $C_k:=\sup_{(\omega,x)\in \Omega\times X} \| D^k\varphi_{\omega}(x)\|$.
  Now,  assume that \eqref{derivativeconvex2} holds for $k \in \N$.  Fix $\varepsilon >0 $ and consider a sequence of function $(h_j) \subset L^p(\Omega, X)$ converging to $0$. Then, we can find $j_0 \in \N$ such that for $j\geq j_0$
  \begin{equation*}
  \begin{aligned}
  \mu( A_j )&\leq \varepsilon^{\frac{p}{p-1}} &\textrm{ if }p\in(1,+\infty),\\
  \mu( A_j )& =0 &\textrm{ if }p=\infty, 
  \end{aligned}
  \end{equation*}
  where, $A_j :=\{ \omega \in \Omega :  \| h_j(\omega)\| >\varepsilon\}$.  Now, let us consider the quantities
  \begin{equation*}
  \begin{aligned}
  \beta_j& :=  \left\| D^k\IntfLp{\varphi}(x + h_j) -D^k\IntfLp{\varphi}(x)   - \int_\Omega  D^{k+1}\varphi_\omega(x(\omega)) (h_j(\omega)) \dmu \right\|,\\
  \beta^1_j& := \left\| \int_{A_j}T_{j}^{k}(\omega)\dmu \right\|,\\
  \beta^2_j&:=    \left\| \int_{A^c_j}  T_{j}^{k}(\omega)  \dmu \right\|,
  \end{aligned}
  \end{equation*} 
  where 
  $$
  T_{j}^{k}(\omega):=D^k\varphi_\omega( x(\omega)+h_j(\omega)) - D^k\varphi_\omega(x(\omega)) - D^{k+1}\varphi_\omega(x(\omega)) (h_j(\omega)).
  $$
\emph{Estimation of $\beta_j^1$:} It is clear that  $\beta_j^1=0$ for $p =\infty$. Thus, we focus on the case $p\in (1,+\infty)$. By using \eqref{eq00main}, we get
  \begin{equation*}
  \begin{aligned}
  \beta^1_j  \leq& \left\| \int_{A_j}  \left(   D^k\varphi_\omega( x(\omega)+h_j(\omega)) - D^k\varphi_\omega( x(\omega)) \right)   \dmu \right\| \\&+ \left\| \int_{A_j}  D^{k+1}\varphi_\omega( x(\omega))(h_j(\omega))  \dmu \right\|\\
  \leq &C_{k+1}  \int_{A_j}  \| h_j(\omega) \|\dmu  +  C_{k+1}   \int_{A_j}  \| h_j(\omega) \|\dmu \\
  \leq & 2C_{k+1} \| h_j\|_p  \mu(A_j)^{\frac{p-1}{p}}\leq 2C_{k+1} \| h_j\|_p \cdot \varepsilon.
  \end{aligned}
  \end{equation*}
  \emph{Estimation of $\beta_j^1$:} By using \eqref{taylor}, we obtain that
  \begin{equation*}
  \beta_j^2\leq C_{k+2} \mu(\Omega)^{\frac{p-1}{p}} \| h_j\|_p \cdot \varepsilon.
  \end{equation*}
  Hence, by using the latter estimations,
  \begin{equation*}
  \frac{\beta_j}{\| h_j\|_p }  \leq \left(  2C_{k+1}   + C_{k+2} \mu(\Omega)^{\frac{p-1}{p}} \right)\cdot  \varepsilon \, \textrm{ for all } j \geq j_0.
  \end{equation*}
  Therefore, we conclude that
  \begin{align*}
\frac{1}{\| h_j\|_p} \left( D^k\IntfLp{\varphi}(x + h_j) -D^k\IntfLp{\varphi}(x )   - \int_\Omega  D^{k+1}\varphi_\omega(x(\omega)) (h_j(\omega)) \dmu \right) \to0,
  \end{align*} as $j\to +\infty$.  Consequently, \eqref{derivativeconvex2} holds.
	\end{proof}
\begin{remark}
Differentiability and subdifferentibility of integral functions has been studied by several authors (see, e.g., \cite{MR4261271,MR3947674,MR4062793,MR3783778,Mordukhovich2021,MR310612,MR236689,MR3639281,MR1618939} and the references therein). It can be shown that the integral function $\IntfLp{\varphi}$, where $\varphi$ is the function obtained from Theorem \ref{mainconvex}, is continuously differentiable over $L^1(\Omega, X)$ and its derivative has the following integral representation
	\begin{equation*}
	D\IntfLp{\varphi}(x)=\int_\Omega D\varphi_\omega(x(\omega))\dmu \textrm{ for } x\in  L^1(\Omega, X).
	\end{equation*}
	Nevertheless, in \cite[Theorem 4.7]{MR2817485}, it is shown that differentiability of integral functionals over $L^1(\Omega, \R^d)$ is related to the convexity of the integrand. Thus, we cannot expect high-order differentiability of the integral $\IntfLp{\varphi}$ over $L^1(\Omega, \R^d)$ together with an integral representation for its $k$-derivatives unless 	$D\varphi_\omega $  is  linear. Indeed, we observe that if $\IntfLp{\varphi}$ is  $C^2$, then for any $\lambda^\ast\in L^\infty(\Omega, \R^d)$ the integral function
	$$ x    \to  \langle \lambda^\ast, 	D\IntfLp{\varphi}(x) \rangle = \int_\Omega \langle \lambda^\ast(\omega), D\varphi_\omega(\omega) \rangle \dmu, $$
	is $C^1$ over $L^1(\Omega, \R^d)$ and by virtue of \cite[Theorem~4.7]{MR2817485}, it must be convex, which is not true for general convex normal integrands.
	\end{remark}
We end this section, by showing that for a general measurable multifunction the convex closure of the set of $p$-integrable measurable selections is equal to the set of minimizers of a convex integral functional, which, for $p\in [1,+\infty)$,  is $C^\infty$.  
\begin{theorem} Let $(\Omega, \A,\mu)$ be a finite measure space. 
 Let  $M: \Omega \tto X$ be a measurable multifunction  with nonempty and closed values such that for $p\in [1,+\infty)$ the set of $p$-integrable selections $S_M^p$ is nonempty. Then, there exists a $C^\infty$ convex  normal integrand function  $\varphi :\Omega \times X\to [ 0, +\infty)$  such that 
 \begin{align}\label{eqmainconvex0}
 \cl\co	S_M^p = \left\{  x\in L^p(\Omega ,X) \colon \IntfLp{\varphi}(x)=0    \right\},
 \end{align}
 where $\cl\co S_M^p$ denotes the closed convex hull of $S_M^p$.  In addition, if   $(\Omega,\mathcal{A},\mu)$ is non-atomic, then 
\begin{align}\label{eqmainconvex00}
	\cl^{w}\left(	S_M^p \right) = \left\{  x\in L^p(\Omega ,X) \colon \IntfLp{\varphi}(x)=0    \right\},
\end{align}
where $\cl^{w}$ denotes the closure with respect to the weak topology on $L^p(\Omega, X)$.
 Moreover, in both cases, it is possible to choose $\varphi$ satisfying the estimations \eqref{eq00main}.
	\end{theorem}
\begin{proof} Due to \cite[Proposition 2.26]{MR1485775}, we observe that $\cl\co M$ is measurable. On the one hand, by  virtue of \cite[Proposition 6.4.17]{MR2527754} (or  \cite[Ch. 2, Proposition 3.29]{MR1485775}),  we have that 
	 $ \cl\co	S_M^p  = S_{ \cl\co M}^p$. Then, by applying Theorem \ref{mainconvex} to the measurable multifunction $\cl\co M$, we obtain  \eqref{eqmainconvex0}. On the other hand, \cite[Proposition 6.4.19]{MR2527754} implies that $\cl^{w}S_M^p  = S_{\cl\co M}^p$ provided $\mu$ is non-atomic. Finally, \eqref{eqmainconvex00}  follows in a similar way.
	 \end{proof}

\section{Examples in Optimization Theory}\label{SectionExamples} 

This section provides two examples where our results can be used to provide methodologies to solve challenging optimization problems.

The first example relies on the setting of Mathematical programming with equilibrium constraints. This class of problems has captured the attention of several researchers due to its intrinsic relation with  Nash equilibrium and bilevel problems (see, e.g, \cite{MR3753620,MR3938480,MR2886139}).

\begin{example}[Mathematical programming with equilibrium constraints] Consider  the following mathematical program with equilibrium constraints:
\begin{equation}\label{problemexample}
    \min \psi(x,y) \text{ s.t. } y\in M(x)  \textrm{  and } x\in C,
\end{equation}
where $\psi:  \mathbb{R}^s \times \mathbb{R}^m \to \R$ is the objective function, $C\subseteq \mathbb{R}^s $ is a closed set, and $M: \mathbb{R}^s \tto \mathbb{R}^m$ is a  multifunction. It has been shown that optimality conditions for this class of problems can be written in terms of generalized Fermat's rules involving (sub)-gradients of $\psi$ and  normal cones to $\gph M$ and $C$ (see, e.g., \cite[Chapter 5.2.1]{MR2191745} and references therein). Due to Corollary \ref{CorAprRandomMult}, and under its assumptions,  it is possible to approximate Problem \eqref{problemexample} changing the constraint $y\in M(x)$ by $y\in M^k(x)$. Furthermore, it follows from  assertion d) from Corollary \ref{CorAprRandomMult} that $\gph M^k$ is an $C^\infty$-smooth manifold. So the computation to the normal cone follows from the determination of such smooth representation (see, e.g., \cite[ Example 6.8]{MR1491362}).
\end{example}

The second example corresponds to the setting of two-stage stochastic programming. This kind of model is a stochastic programming problem where the intention is to chose an initial vector decision, then after the realization of a random phenomenon, the information is included, and a second choice must be given, all of this with a minimal cost. We refer to the monographs \cite{MR3242164,MR1375234} for more details about the theory.  
 \begin{example}[Two-stage stochastic programming]
 	Let $(\Omega, \mathcal{A},\mathbb{P})$ be a probability space, and consider the abstract two-stage stochastic convex  optimization problem
 	\begin{equation}\label{2stage}
 	\begin{aligned}
 		&\min \IntfLp{\psi}(x,y):= \int_{ \Omega} \psi (\omega, x,y(\omega)) \mathbb{P}(d \omega ) \\
 		&\text{ s.t } y(\omega)\in M(\omega,x ) \text{ a.s } \omega \in \Omega, 
 		x\in \mathbb{R}^s \text{ and } y\in L^p(\Omega,\R^m),
 	\end{aligned}	
 	 	\end{equation}
 	where $M : \Omega \times \mathbb{R}^s  \tto\mathbb{R}^m $ is a random multifunction and $\psi : \Omega \times \mathbb{R}^s \times \mathbb{R}^m \to \R$ is a normal integrand, which is convex with respect to the variable on $\R^m$ and satisfies \eqref{eqconjfenchel}. First, let us mention that the imposed convexity on the normal integrand $\psi$. On the one hand, it is essential to provide the strong-weak lower-semicontinuity of the integral function $\IntfLp{\psi}$  (see, e.g., \cite[Theorem 2.1]{MR917861}), which is a minimal requirement to   obtain the existence of minimizers of Problem \eqref{2stage}. On the other hand, it is important to provide the stability of such minimizers (see, e.g., \cite[Theorem 5.1]{Mordukhovich2021}. For similar purposes,  the convexity of the values of $M$ in the optimization problem \eqref{2stage} is required.
 	
 	Now, by virtue of Theorem \ref{thmfuncap}, it is possible to find a nondecreasing   sequence of $C^\infty$-normal integrands $(\psi^k)_{k}$ converging to $\psi$. Moreover, due to Corollary \ref{CorAprRandomMult}, there exists a decreasing sequence of random multifunctions $(M^k)_k$. Both sequences preserve the previously discussed important property of convexity.  Hence, we can define a sequence of optimization problems:
 		\begin{equation*} 
 	\begin{aligned}
 		&\min \IntfLp{\psi^k}(x,y):= \int_{ \Omega} \psi^k (\omega, x,y(\omega)) \mathbb{P}( d\omega ) \\
 		&\text{ s.t }  y(\omega)  \in M^k(\omega,x ) \text{ a.s } \omega \in \Omega,   			x\in \mathbb{R}^s \text{ and } y\in L^p(\Omega,\R^m).
 	\end{aligned}	
 	 	\end{equation*}
 	 	Here, it is important to emphasize that by the construction, the optimal value of the above problems is  an increasing sequence, which under some classical compactness assumptions it converges to the optimal value of Problem \eqref{2stage}. In the same spirit, by using classical techniques,  the minimizers  should converge (up to a subsequence) to a minimizer of the original problem \eqref{2stage}.
 	\end{example}

\section{Concluding remarks}

In this paper, under mild continuity assumptions, we prove that any  random multifunction can be represented as the set of minimizers of an infinitely many differentiable normal integrand. This result was used to deduce the existence of smooth approximations for multifunctions and normal integrands. Morever, we characterize the set of $p$-integrable selections of any measurable multifunctions can be represented as the set of minimizers of infinitely many differentiable integral functional. 

The results obtained in this article offer several tools for the representation and approximation of multifunctions and normal integrands by smoother objects, which is of importance for approximation in optimization, optimal control and other areas.

\bibliographystyle{plain}
\bibliography{references}
\end{document}